\documentclass[11pt,a4paper]{article}

\usepackage{inputenc}
\usepackage{amsmath}
\usepackage{bm}
\usepackage{bbold}
\usepackage{amsthm}

\usepackage{hyperref}

\title{Bounds on the State Vector Growth Rate\\ in Stochastic Dynamical Systems\thanks{Proc. 5th St.~Petersburg Workshop on Simulation, St. Petersburg, Russia, June 26-July 2, 2005 / Ed. by S.~M.~Ermakov and V.~B.~Melas and A.~N.~Pepelyshev, St.~Petersburg, 2005, pp.~391-396.}} 

\author{Nikolai~K.~Krivulin\thanks{Faculty of Mathematics and Mechanics, St.~Petersburg State University, 28 Universitetsky Ave., St.~Petersburg, 198504, Russia, 
nkk@math.spbu.ru.} \thanks{The work was partially supported by the Russian Foundation for Basic Research, Grant \#04-01-00840.}
}

\date{}

\newtheorem{theorem}{Theorem}
\newtheorem{lemma}[theorem]{Lemma}
\newtheorem{corollary}[theorem]{Corollary}

\begin{document}

\maketitle

\begin{abstract}
A stochastic dynamical system represented through a linear vector equation in idempotent algebra is considered. We propose simple bounds on the mean growth rate of the system state vector, and give an analysis of absolute error of a bound. As an illustration, numerical results of evaluation of the bounds for a test system are also presented.
\\

\end{abstract}

\section{Introduction}

The evolution of actual systems encountered in economics, management, and engineering can frequently be represented as stochastic linear dynamic equations in idempotent algebra \cite{Baccelli1993Synchronization,Maslov1994Idempotent}. In many cases, analysis of the system can involve evaluation of the asymptotic growth rate of the system state vector. However, the exact evaluation of the growth rate normally appears to be a hard problem. The exact solution is known only for systems with $2$-dimensional state space \cite{Baccelli1993Synchronization}, systems with a triangular state transition matrix \cite{Krivulin2005Thegrowth}, and some others.

In this paper, we propose simple bounds on the asymptotic (mean) growth rate, which can be considered as a generalization for bounds in \cite{Krivulin2003Estimation}. We start with a brief overview of related algebraic results including some matrix inequalities. Based on these inequalities, both upper and low bounds are derived, and an analysis of the absolute error of an upper bound is given. As an illustration, numerical results of evaluation of the bounds for a test system are also presented.

\section{Idempotent Algebra and Related Results}
We consider an idempotent algebra (idempotent semifield with a null element) with addition $ x\oplus y=\max(x,y) $ and multiplication $ x\otimes y=x+y $ defined for all $ x $ and $ y $ from the extended set of real numbers $ \mathbb{R}_{\varepsilon}=\mathbb{R}\cup\{\varepsilon\} $, where $ \varepsilon=-\infty $.

Clearly, the numbers $ \varepsilon $ and $ 0 $ present null and identity elements of the algebra. For any $ x\in\mathbb{R} $, one can define its inverse $ x^{-1} $ that is equal to $ -x $ in conventional algebra, and the power $ x^{a} $, which corresponds to arithmetic product $ ax $ for all $ a\in\mathbb{R} $. In the case that $ x=\varepsilon $, it is convenient to set $ x^{-1}=\varepsilon $.

The matrix operations $ \oplus $ and $ \otimes $ are introduced in the usual way through their scalar counterparts. The matrix $ \cal{E} $ involving only $ \varepsilon $ presents the null matrix, and $ E=\mathrm{diag}(0,\ldots,0) $ with all off-diagonal entries being equal to $ \varepsilon $ is the identity.

Any nonnegative integer power of a square matrix $ A $ is determined by the relations: $ A^{0}=E $ and $ A^{l}\otimes A^{m}=A^{l+m} $ for all integers $ l,m\geq1 $. In what follows, the exponential notations will be used only in the sense of idempotent algebra. However, for simplicity sake, we will sometimes represent the power of a number in the form of its equivalent arithmetic product.

For any matrix $ A=(a_{ij}) $, we introduce the matrix $ A^{-} $ with entries $ a_{ij}^{-}=a_{ji}^{-1} $. Similarly, for any vector $ \bm{x}=(x_{1},\ldots,x_{n})^{T} $, we have $ \bm{x}^{-}=(x_{1}^{-1},\ldots,x_{n}^{-1}) $.

The operation $ \otimes $ is monotonic; that is, from the inequalities $ A\leq C $ and $ B\leq D $, it follows that $ A\otimes B\leq C\otimes D $.

For any matrices $ A\in\mathbb{R}_{\varepsilon}^{n\times n} $ and $ B\in\mathbb{R}^{n\times n} $, and $ \bm{0}=(0,\ldots,0)^{T} $, it holds
\begin{equation}\label{I-ABA0B0}
A\otimes B
\geq
A\otimes\bm{0}\otimes(B^{-}\otimes\bm{0})^{-}.
\end{equation}

Consider a matrix $ A=(a_{ij})\in\mathbb{R}_{\varepsilon}^{n\times n} $, and introduce the symbols:
$$
\|A\|=\bigoplus_{1\leq i,j\leq n}a_{ij},
\qquad
\mathrm{tr}(A)=\bigoplus_{i=1}^{n}a_{ii}.
$$

For any $ A,B\in\mathbb{R}_{\varepsilon}^{n\times n} $, if $ A\leq B $ then $ \|A\|\leq\|B\| $. Furthermore, it holds that
$$
\|A\otimes B\|
\leq
\|A\|\otimes\|B\|,
\qquad
\|c\otimes A\|=c\otimes\|A\|
\quad
\mbox{for all $ c\in\mathbb{R}_{\varepsilon}$}.
$$

If $ A\in\mathbb{R}_{\varepsilon}^{n\times n} $ and $ B\in\mathbb{R}^{n\times n} $, we also have
$$
\|A\otimes B\|
\geq
\|A\|\otimes\|B^{-}\|^{-1}.
$$

The key result of the spectral theory in idempotent algebra is as follows \cite{Romanovskii1967Optimization,Vorobjev1967Extremal}: for any matrix $ A\in\mathbb{R}_{\varepsilon}^{n\times n} $, it holds that
\begin{equation}\label{E-limAkk}
\lim_{k\to\infty}\|A^{k}\|^{1/k}
=
\rho(A)
=
\bigoplus_{m=1}^{n}\mathrm{tr}^{1/m}(A^{m}),
\end{equation}
where $ \rho(A) $ is the spectral radius of $ A $.

Let us now consider random matrices taking their values in $ \mathbb{R}_{\varepsilon}^{n\times n} $. For any random matrix $ A $, we use the symbol $ \mathbb{E}[A] $ to denote the matrix obtained from $ A $ by replacing all its entries with their expected values, provided that $ \mathbb{E}[\varepsilon]=\varepsilon $.

For any random matrices $ A $ and $ B $, it holds that
$$
\mathbb{E}\|A\|
\geq
\|\mathbb{E}[A]\|,
\quad
\mathbb{E}[A\otimes B]
\geq
\mathbb{E}[A]\otimes\mathbb{E}[B].
$$

Furthermore, let the matrices $ A $ and $ B $ be independent. Then we have
$$
\mathbb{E}\|A\otimes B\|
\geq
\mathbb{E}\|A\otimes\mathbb{E}[B]\|.
$$
If, in addition, the entries of $ B $ are finite with probability 1 (w.p.~1), then
\begin{equation}\label{I-EABEAEB0}
\mathbb{E}\|A\otimes B\|
\geq
\mathbb{E}\|A\|\otimes\|\mathbb{E}(B\otimes\bm{0})^{-}\|^{-1}
\geq
\mathbb{E}\|A\|\otimes\|\mathbb{E}[B^{-}]\|^{-1}.
\end{equation}

\section{Stochastic Dynamical Systems}

Consider a dynamical system governed by the equation
\begin{equation}\label{E-xkAkxk1}
\bm{x}(k)
=
A^{T}(k)\otimes\bm{x}(k-1),
\end{equation}
where $ \bm{x}(k) $ is a state vector, $ A(k) $ is a random state transition matrix.

We assume that the matrices $ A(k) $, $ k=1,2,\ldots, $ are independent and identically distributed, and that the mean value $ \mathbb{E}\|A(1)\| $ is finite.

Let us define the mean (asymptotic) growth rate of the system state vector as
$$
\lambda=\lim_{k\to\infty}\|\bm{x}(k)\|^{1/k}.
$$

Assuming the entries of the initial vector $ \bm{x}(0) $ to be finite w.p.~1, one can represent $ \lambda $ in the form
$$
\lambda=\lim_{k\to\infty}\|A_{k}\|^{1/k},
$$
where
$$
A_{k}=A(1)\otimes\cdots\otimes A(k).
$$

It can be shown (e.g., with the ergodic theorem in \cite{Kingman1973Subadditive}) that for the system under consideration, the above limit exists w.p.~1. Moreover, there exists the limit
$$
\lim_{k\to\infty}\mathbb{E}\|A_{k}\|^{1/k}=\lambda.
$$

The last result will be used in subsequent sections to derive bounds on $ \lambda $.

As an example, we consider a test system (\ref{E-xkAkxk1}) with random $(2\times 2)$-matrix $ A(k) $ with independent entries, each having the exponential probability distribution of mean $1$. It is known (see, e.g. \cite{Baccelli1993Synchronization}) that for the system, $ \lambda=407/228\approx1.7851 $.

To illustrate the bounds presented bellow, we need to know the means of the entries $ (A_{m})_{ij} $, row maxima $ (A_{m}\otimes\bm{0})_{i} $, and the overall maximum $ \|A_{m}\| $ of the matrix $ A_{m} $. Evaluation of the means for $ m=1,2,3 $, gives us the following results
\begin{gather*}
\mathbb{E}[(A_{1})_{ij}]=1,
\qquad
\mathbb{E}[(A_{1}\otimes\bm{0})_{i}]=1.5,
\qquad
\mathbb{E}\|A_{1}\|=\frac{25}{12}\approx2.0833, \\
\mathbb{E}[(A_{2})_{ij}]=2.75,
\quad
\mathbb{E}[(A_{2}\otimes\bm{0})_{i}]=\frac{119}{36}\approx3.3056,
\quad
\mathbb{E}\|A_{2}\|=\frac{833}{216}\approx3.8565, \\
\mathbb{E}[(A_{3})_{ij}]=\frac{245}{54}\approx4.5370,
\qquad
\mathbb{E}[(A_{3}\otimes\bm{0})_{i}]=\frac{1649}{324}\approx5.0895, \\
\mathbb{E}\|A_{3}\|=\frac{21937}{3888}\approx5.6422.
\end{gather*}

Note that it is easy to get the means when $ m\leq2 $. However, the evaluation rapidly grows in computational complexity as $ m $ becomes greater than $ 2 $.

\section{Straightforward Low and Upper Bounds}

We start with simple low and upper bounds which are valid for systems with any matrix $ A_{1} $ having a finite mean value $ \mathbb{E}\|A_{1}\| $. 
\begin{lemma}
For any integer $ m\geq1 $, it holds
\begin{equation}\label{I-LU1}
\rho^{1/m}(\mathbb{E}[A_{m}])
\leq
\lambda
\leq
\mathbb{E}\|A_{m}\|^{1/m}.
\end{equation}
\end{lemma}
\begin{proof} In order to verify (\ref{I-LU1}), let us first put $ m=1 $, and note that 
$$
\|(\mathbb{E}[A_{1}])^{k}\|
=
\left\|\bigotimes_{i=1}^{k}\mathbb{E}[A(i)]\right\|
\leq
\mathbb{E}\|A_{k}\|
\leq
\mathbb{E}\left[\bigotimes_{i=1}^{k}\|A(i)\|\right]
=
\mathbb{E}\|A_{1}\|^{k}.
$$

It remains to divide the above double inequality by $ k $, and proceed to get limits. With (\ref{E-limAkk}) applied to the left side, we immediately arrive at
$$
\rho(\mathbb{E}[A_{1}])
\leq
\lambda
\leq
\mathbb{E}\|A_{1}\|.
$$

The case of arbitrary $ m>1 $ can be considered in a similar way.
\end{proof}

Table~\ref{T-LU1} presents results of evaluating the bounds for the test problem.
\begin{table}[!ht]
\begin{center}
\begin{tabular}{||c|c|c|c||}
\hline\hline
Bounds & \multicolumn{3}{|c||}{$m$} \\
\cline{2-4}
(\ref{I-LU1}) & $1$ & $2$ & $3$ \\
\cline{2-4} 
\hline
Upper & 2.0833 & 1.9282 & 1.8807 \\
\hline
Low & 1.0000 & 1.3750 & 1.5123 \\
\hline\hline
\end{tabular}
\caption{Bounds evaluated according to (\ref{I-LU1}).}\label{T-LU1}
\end{center}
\end{table}

\section{Low Bounds for Finite Matrices}

Suppose now that all the entries of the matrix $ A_{1} $ are greater than $ \varepsilon $ w.p.~1. 

\begin{lemma}
For any integer $ m\geq1 $, it holds
\begin{equation}\label{I-L2}
\lambda
\geq
\|\mathbb{E}(A_{m}\otimes\bm{0})^{-}\|^{-1/m}.
\end{equation}
\end{lemma}

\begin{proof} In the case that $ m=1 $, we apply (\ref{I-EABEAEB0}) to write
$$
\mathbb{E}\|A_{k}\|
\geq
\mathbb{E}\|A_{1}\|\otimes\|\mathbb{E}(A_{1}\otimes\bm{0})^{-}\|^{-(k-1)},
$$
and then get the inequality
$$
\lambda
\geq
\|\mathbb{E}(A_{1}\otimes\bm{0})^{-}\|^{-1},
$$
which can easily be extended to arbitrary integer $ m\geq1 $ in the form of (\ref{I-L2}).
\end{proof}

Evaluation of the bounds for $ m=1,2,3 $, gives us: $ 1.5000 $, $ 1.6528 $, and $ 1.6965 $.

\begin{lemma}
For any integers $ l,m\geq1 $, it holds
\begin{equation}\label{I-L3}
\lambda
\geq
\mathbb{E}\|(\mathbb{E}[A_{l}^{-}]\otimes\bm{0})^{-}\otimes A_{m}\|^{1/(l+m)}.
\end{equation}
\end{lemma}

\begin{proof} Let us prove the inequality for $ l=m=1 $. Setting $ k=2s $, one can apply (\ref{I-ABA0B0}) to get
\begin{multline*}
\mathbb{E}\|A_{k}\|
\geq
\mathbb{E}\left\|\bigotimes_{i=1}^{s}A(2i-1)\otimes\bm{0}\otimes(\mathbb{E}[A^{-}(2i)]\otimes\bm{0})^{-}\right\| \\
=
\mathbb{E}\|A(1)\otimes\bm{0}\|
\otimes
\mathbb{E}\left[\bigotimes_{i=1}^{s-1}(\mathbb{E}[A^{-}(2i)]\otimes\bm{0})^{-}\otimes(A(2i+1)\otimes\bm{0})\right] 
\\
\otimes
\|(\mathbb{E}[A^{-}(k)]\otimes\bm{0})^{-}\| \\
=
\mathbb{E}\|A_{1}\otimes\bm{0}\|
\otimes
\mathbb{E}[(\mathbb{E}[A_{1}^{-}]\otimes\bm{0})^{-}\otimes A_{1}\otimes\bm{0}]^{s-1}
\otimes
\|(\mathbb{E}[A_{1}^{-}]\otimes\bm{0})^{-}\|.
\end{multline*}

The last inequality leads us to
$$
\lambda
\geq
\mathbb{E}[(\mathbb{E}[A_{1}^{-}]\otimes\bm{0})^{-}\otimes A_{1}\otimes\bm{0}]^{1/2}
=
\mathbb{E}\|(\mathbb{E}[A_{1}^{-}]\otimes\bm{0})^{-}\otimes A_{1}\|^{1/2}.
$$

In a similar manner, inequality (\ref{I-L3}) can be derived for arbitrary $ l,m\geq1 $.
\end{proof}

Examples related to evaluation of the low bound are presented in Table~\ref{T-LB3}.
\begin{table}[!ht]
\begin{center}
\begin{tabular}{||c|c|c|c|c||}
\hline\hline
\multicolumn{2}{||c|}{Low Bound} & \multicolumn{3}{|c||}{$m$} \\
\cline{3-5}
\multicolumn{2}{||c|}{(\ref{I-L3})} & $1$ & $2$ & $3$ \\
\cline{3-5} 
\hline
& $1$ & 1.5417 & 1.6188 & 1.6606 \\
\cline{2-5}
$l$ & $2$ & 1.6111 & 1.6516 & 1.6784 \\
\cline{2-5}
& $3$ & 1.6551 & 1.6787 & 1.6965 \\
\hline\hline
\end{tabular}
\caption{Examples of low bounds (\ref{I-L3}).}\label{T-LB3}
\end{center}
\end{table}

As one can see, the accuracy of bounds (\ref{I-L2}) when $ m=1,2,3 $, are quite comparable to that of bounds (\ref{I-L3}). Note, however, that in order to achieve the same accuracy, the first bound involves less computational efforts than the second.

We conclude this section with a result to be used in the error analysis below.
\begin{corollary}
For any integer $ m\geq1 $, it holds
\begin{equation}\label{I-L4}
\lambda
\geq
(\|\mathbb{E}[A_{1}^{-}]\|^{-1}\otimes\mathbb{E}\|A_{m-1}\|)^{1/m}.
\end{equation}
\end{corollary}

\begin{proof} In order to verify inequality (\ref{I-L4}) for $ m=1 $, first note that
$$
\mathbb{E}\|A_{k}\|
\geq
\|(\mathbb{E}[A_{1}])^{k}\|
\geq
\|\mathbb{E}[A_{1}^{-}]\|^{-k},
$$
and therefore, $ \lambda\geq\|\mathbb{E}[A_{1}^{-}]\|^{-1}=\|\mathbb{E}[A_{1}^{-}]\|^{-1}\otimes\mathbb{E}\|A_{0}\| $, where $ A_{0}=E $.

For any $ m>1 $, the inequality results from (\ref{I-L3}):
$$
\lambda
\geq
\mathbb{E}\|(\mathbb{E}[A_{1}^{-}]\otimes\bm{0})^{-}\otimes A_{m-1}\|^{1/m}
\geq
(\|\mathbb{E}[A_{1}^{-}]\|^{-1}\otimes\mathbb{E}\|A_{m-1}\|)^{1/m}.
\qedhere
$$
\end{proof}

\section{An Error Analysis}

Consider the absolute error
$$
e_{m}=\frac{1}{m}\mathbb{E}\|A_{m}\|-\lambda,
$$
which represents the difference between $ \lambda $ and its approximation $ \mathbb{E}\|A_{m}\|^{1/m} $.

Assuming the entries of $ A_{1} $ to be finite w.p.~1, we have the following result.
\begin{lemma}
For each $ m\geq1 $, it holds
$$
e_{m}
\leq
\frac{1}{m}C,
$$
where $ C=\mathbb{E}\|A_{1}\|+\|\mathbb{E}[A_{1}^{-}]\|=\mathbb{E}\|A_{1}\|\otimes\|\mathbb{E}[A_{1}^{-}]\| $.
\end{lemma}
\begin{proof}
By applying the obvious inequality: $ \mathbb{E}\|A_{m}\|\leq\mathbb{E}\|A_{1}\|\otimes\mathbb{E}\|A_{m-1}\| $, combined with (\ref{I-L4}), we have
\begin{multline*}
e_{m}
\leq
\frac{1}{m}(\mathbb{E}\|A_{1}\|+\mathbb{E}\|A_{m-1}\|)
-
\frac{1}{m}(\|\mathbb{E}[A_{1}^{-}]\|^{-1}+\mathbb{E}\|A_{m-1}\|) \\
=
\frac{1}{m}(\mathbb{E}\|A_{1}\|+\|\mathbb{E}[A_{1}^{-}]\|)
=
\frac{1}{m}(\mathbb{E}\|A_{1}\|\otimes\|\mathbb{E}[A_{1}^{-}]\|).
\qedhere
\end{multline*}
\end{proof}

Let us compute the constant $ C $ for the test problem. Since $ \mathbb{E}\|A_{1}\|=25/12 $, $ \|\mathbb{E}[A_{1}^{-}]\|=-1 $, we have $ C=13/12\approx1.0833 $. 

As it is easy to see, in the test example, the above error bound considerably overestimates the actual error at least for $ m=1,2,3 $.

\bibliographystyle{utphys}

\bibliography{Bounds_on_the_state_vector_growth_rate_in_stochastic_dynamical_systems}

\providecommand{\href}[2]{#2}\begingroup\raggedright\begin{thebibliography}{1}

\bibitem{Baccelli1993Synchronization}
F.~L. Baccelli, G.~Cohen, G.~J. Olsder, and J.-P. Quadrat, {\em Synchronization
  and Linearity: An Algebra for Discrete Event Systems}.
\newblock Wiley Series in Probability and Statistics. Wiley, Chichester, 1993.
\newblock \url{http://www-rocq.inria.fr/metalau/cohen/documents/BCOQ-book.pdf}.

\bibitem{Maslov1994Idempotent}
V.~Maslov and V.~Kolokoltsov, {\em Idempotent Analysis and Its Applications to
  Optimal Control Theory}.
\newblock Nauka, Moscow, 1994.
\newblock (in Russian).

\bibitem{Krivulin2005Thegrowth}
N.~K. Krivulin, ``The growth rate of state vector in a generalized linear
  dynamical system with random triangular matrix,'' {\em Vestnik St. Petersburg
  Univ. Math.} {\bfseries 38} no.~1, (March, 2005) 25--28.

\bibitem{Krivulin2003Estimation}
N.~K. Krivulin, ``Estimation of the rate of growth of the state vector of a
  generalized linear dynamical system with random matrix,'' {\em Vestnik St.
  Petersburg Univ. Math.} {\bfseries 36} no.~3, (September, 2003) 41--49.

\bibitem{Romanovskii1967Optimization}
I.~Romanovskii, ``Optimization of stationary control of a discrete
  deterministic process,'' \href{http://dx.doi.org/10.1007/BF01078754}{{\em
  Cybernetics} {\bfseries 3} (April, 1967) 52--62}.

\bibitem{Vorobjev1967Extremal}
N.~N. Vorobjev, ``Algebra of positive matrices,'' {\em Elektronische
  Informationsverarbeitung und Kybernetik} {\bfseries 3} no.~1, (February,
  1967) 39--71. (in Russian).

\bibitem{Kingman1973Subadditive}
J.~F.~C. Kingman, ``Subadditive ergodic theory,'' {\em Ann. Probab.} {\bfseries
  1} no.~6, (December, 1973) 883--899.

\end{thebibliography}\endgroup

\end{document}